\documentclass[reqno, 12pt]{amsart}
\usepackage{enumerate}
\usepackage{amssymb,url,color}
\usepackage{hyperref}
\usepackage{amsmath,amsthm}

\allowdisplaybreaks[4]

\newtheorem{thm}{Theorem}[section]

\newtheorem{lem}{Lemma}[section]
\newtheorem{rem}{Remark}[section]
\newtheorem{exa}{Example}[section]
\theoremstyle{Problem}

\theoremstyle{definition}
\newtheorem{defn}{Definition}[section]
\numberwithin{equation}{section}
\newcommand{\pp}{\mathbb{P}}

\newcommand{\ee}{\mathbb{E}}

\newcommand{\FF}{\mathcal{F}}

\newcommand{\rr}{\mathbb{R}}
\newcommand{\ii}{\mathbb{I}}

\def\beq{\begin{equation}}
\def\deq{\end{equation}}
\def\dsp{\displaystyle}
\allowdisplaybreaks 

\bfseries\rmfamily 

\begin{document}
\title[MDP for the chi-square statistics]
{Moderate deviation principle for the chi-square statistics}
\thanks{This work is supported by National Natural Science Foundation of China (NSFC-11971154).}

\author[Z. H. Yu]{Zhenhong Yu}
\address[Z. H. Yu]{School of Mathematics and Statistics, Henan Normal University, Henan Province, 453007, China.} \email{\href{mailto: Z. H. Yu
<zhenhongyu2022@126.com>}{zhenhongyu2022@126.com}}

\author[Y. Miao]{Yu Miao}
\address[Y. Miao]{School of Mathematics and Statistics, Henan Normal University, Henan Province, 453007, China.} \email{\href{mailto: Y. Miao
<yumiao728@gmail.com>}{yumiao728@gmail.com}; \href{mailto: Y. Miao <yumiao728@126.com>}{yumiao728@126.com}}

\begin{abstract}
In the present paper, we consider the Pearson chi-square statistic defined on a finite alphabet which is assumed to dynamically vary as the sample size increases, and establish its moderate deviation principle.
\end{abstract}

\keywords{Pearson chi-square statistic, moderate deviation principle.}
\subjclass[2020]{60F10, 62G30, 62G20}

\maketitle
\section{Introduction}
The Pearson chi-square statistic stands as one of the most renowned and significant contributions to statistical science, maintaining a pivotal role in statistical applications since its initial introduction in Pearson's seminal work on randomness testing (see \cite{Pearson}).
The Pearson chi-square goodness-of-fit test inherently presumes that the support of the target discrete distribution remains invariant (regardless of being finite or infinite) and is unaltered by the sampling procedure.
However, this conventional assumption often proves inadequate for contemporary big-data applications characterized by intricate, adaptive sampling mechanisms. In such scenarios, the statistical testing framework would be more appropriately modeled using triangular arrays of discrete distributions, where both the finite supports and their dimensionality evolve with sample size.

For every $n\ge 1$, let $\{X_{k,n}, 1\le k\le  n\}$ be a sequence of independent and identically distributed random variables with the same distribution as $X_{1,n}$, where
$$
\pp(X_{1,n}=i)=p_n(i), \ \  \ i=1, 2, \dots, m_n.
$$
Recall that the standard Pearson chi-square statistic is defined as
\beq\label{chi}
\chi_n^2=n\sum_{i=1}^{m_n}\frac{(\hat p_n(i)-p_n(i))^2}{p_n(i)},
\deq
where the empirical frequencies $\hat p_n(i)$ are
$$
\hat p_n(i)=\frac{1}{n}\sum_{k=1}^n\ii\{X_{k,n}=i\}, \ \  \ i=1, 2, \dots, m_n.
$$

There are many publications dealing with the asymptotic properties of both the Pearson statistic and its various modifications.
Tumanyan \cite{Tumanyan54, Tumanyan56} proved the asymptotic normality of $\chi^2_n$ under the assumption $\dsp\min_{1\le i\le m_n}np_n(i)\to\infty$, which in the case of the uniform distribution is equivalent to $n/m_n\to\infty$. Gihman \cite{Gihman56} generalized Tumanian's results by considering the case that $\dsp\sum_{i=1}^{m_n}p_n(i)=\alpha_0\le 1$, $\dsp\max_{1\le i\le m_n}p_n(i)\to0$ and $\dsp\min_{1\le i\le m_n}np_n(i)\to\infty$, where $\alpha_0$ depends on $n$. Some theorem were given for sequences of series of certain statistics similar in structure to $\chi^2$, which show that these sequences converge to corresponding continuous Markov processes.
Holst \cite{Holst72} studied the asymptotic normality of $\chi^2_n$ under the regime $n/m_n\to\lambda\in(0,\infty)$ and $\dsp\max_{1\le i\le m_n}p_n(i)<\beta/n$ for some $\beta>0$.
 Morris \cite{Morris75} gave the asymptotic normality of $\chi^2_n$ under the assumptions $\dsp\min_{1\le i\le m_n}np_n(i)>\varepsilon>0$ for all $n\ge 1$, $\dsp\max_{1\le i\le m_n}p_n(i)\to0$ and some ``uniform asymptotically negligible" condition.
 Ermakov \cite{Ermakov98} proved that a sequence of chi-square tests is
asymptotically minimax if $m_n=o(n^2)$ as $n\to\infty$.
Kruglov \cite{Kruglov01} obtained some limit theorems for some functionals of the Pearson statistic constructed from the polynomial distribution under the conditions that $\dsp \inf_n\{n\min_{1\le i\le m_n}p_n(i)\}>0$, $\dsp n\min_{i\in W_n}p_n(i)\to \infty$, $N_n/m_n\to 1$ as $n\to\infty$, where $N_n$ is the number of elements in the set $W_n\subset\{1,2,\cdots m_n\}$.
Gy\"{o}rfi \cite{G} discussed Chernoff-type large deviation results for $\chi^2$ divergence errors on partitions. In contrast to the total variation and the $I$-divergence, the $\chi^2$-divergence has an unconventional large deviation rate.
Rempa{\l}a and Weso{\l}owski \cite{R-W} established the distributional limit of the Pearson chi-square statistic when the number of classes $m_n$ increases with the sample size $n$ and $n/\sqrt{m_n}\to\lambda$. Under mild moment conditions, the limit is Gaussian for $\lambda=\infty$, Poisson for finite $\lambda>0$, and degenerate for $\lambda=0$.

Based on the above works, in the present paper, we shall continue to study the asymptotic behaviors for the Pearson chi-square statistic defined in (\ref{chi}), and establish the moderate deviation principle of the
chi-square statistic.

\section{Main results}
The following theorem is our main result.
\begin{thm}\label{thm1}
For every $n\ge 1$,
let $\{X_{k,n}, 1\le k\le n\}$ be a sequence of independent and identically distributed random variables with  common nonuniform distribution $\{p_n(i), 1\le i\le m_n\}$, i.e., for any $n\ge1$ and
$1\le k\le n$,
$$
\pp(X_{1,n}=i)=p_n(i), \ \  \ i=1,2,\dots,m_n.
$$
Then for any $r>0$, we have
$$
\lim_{n\to\infty}\frac{1}{b_n^2}\log \pp\left(\frac{1}{b_n\sqrt{2m_n}}\left(\chi_n^2-\ee\chi_n^2\right)>r\right)=-\frac{r^2}{2}
$$
where the sequence $m_n$  and the moderation scale $\{b_n, n\ge 1\}$ are sequences of positive numbers satisfying
\beq\label{con}
\frac{b_n^2}{\log n} \to \infty,\ \ \frac{m_n}{b_n^2}\to \infty,\ \ \frac{n m_n p_{min}^2}{b_n^4}\to\infty,\ \ p_{min}:=\min_{1\le i\le m_n}p_{n}(i).
\deq
\end{thm}
\begin{rem}
Since $p_{min}\le m_n^{-1}$, the condition $\dsp\frac{n m_n p_{min}^2}{b_n^4}\to\infty$ implies $n/(m_nb_n^4)\to\infty$ and $np_{min}/b_n^4\to \infty$. Hence Theorem \ref{thm1} studies the case $n/m_n\to\infty$ and $np_{min}\to\infty$.
\end{rem}

\begin{exa} Let $\alpha\in(0,1)$ and set $p_n(i)=(C_\alpha i^\alpha)^{-1}$ for $i=1,\cdots, m_n$, where
$$
C_\alpha=\sum_{i=1}^{m_n}i^{-\alpha}\sim \frac{m_n^{1-\alpha}}{1-\alpha}.
$$
Then we have $\dsp p_{min}\sim \frac{1-\alpha}{m_n}$. If we choose the sequences $\{b_n, n\ge 1\}$ and $\{m_n, n\ge 1\}$ satisfy
$$
\frac{b_n^2}{\log n} \to \infty,\ \ \frac{m_n}{b_n^2}\to \infty,\ \ \frac{n}{m_nb_n^4}\to\infty,
$$
Theorem \ref{thm1} holds.
\end{exa}
\begin{exa} Let $p_n(i)=(C_n i)^{-1}$ for $i=1,\cdots, m_n$, where
$$
C_n=\sum_{i=1}^{m_n}\frac{1}{i}\sim \ln m_n.
$$
Then we have $\dsp p_{min}\sim \frac{1}{m_n\ln m_n}$. If we choose the sequences $\{b_n, n\ge 1\}$ and $\{m_n, n\ge 1\}$ satisfy
$$
\frac{b_n^2}{\log n} \to \infty,\ \ \frac{m_n}{b_n^2}\to \infty,\ \ \frac{n}{ m_n b_n^4(\ln m_n)^2}\to\infty,
$$
Theorem \ref{thm1} holds.
\end{exa}
\begin{exa} Let $\alpha > 1$ and set $p_n(i)=(C_\alpha i^\alpha)^{-1}$ for $i=1,\cdots, m_n$, where
$$
C_\alpha=\sum_{i=1}^{m_n}i^{-\alpha}<\infty.
$$
Then we have $\dsp p_{min}=O\left(\frac{1}{m_n^\alpha}\right)$. If we choose the sequences $\{b_n, n\ge 1\}$ and $\{m_n, n\ge 1\}$ satisfy
$$
\frac{b_n^2}{\log n} \to \infty,\ \ \frac{m_n}{b_n^2}\to \infty,\ \ \frac{n}{m_n^{2\alpha-1}b_n^4}\to\infty,
$$
Theorem \ref{thm1} holds.
\end{exa}
\begin{exa}
Consider a sequence of distributions of the form
$$p_{n,1}=\frac{1}{m_n}+\frac{1}{n^\gamma}, \ \ \ p_{n,2}=\frac{1}{m_n}-\frac{1}{n^\gamma},$$
where $\gamma \ge 1$ is a real number and $p_{n,i}=1/m_n$ for all $i=3,4,\cdots$.
If we choose the sequences $\{b_n, n\ge 1\}$ and $\{m_n, n\ge 1\}$ satisfy
$$
\frac{b_n^2}{\log n} \to \infty,\ \ \frac{m_n}{b_n^2}\to \infty,\ \ \frac{n}{m_nb_n^4}\to\infty,
$$
in this case
$$
\dsp p_{min}=\frac{1}{m_n}-\frac{1}{n^\gamma}\sim\frac{1}{m_n},
$$
then Theorem \ref{thm1} holds.
\end{exa}
\begin{exa}
For every $i=1,2, \cdots$, let $p_{n,i}=(1-p_n)^{i-1}p_n$ where $p_n=1-\frac{1}{n^\alpha}$ and $\alpha \in (0,1)$. Then we have
$$
\dsp p_{min}\sim\left(\frac{1}{n^\alpha}\right)^{m_n-1}.
$$
If we choose the sequences $\{b_n, n\ge 1\}$ and $\{m_n, n\ge 1\}$ satisfy
$$
\frac{b_n^2}{\log n} \to \infty,\ \ \frac{m_n}{b_n^2}\to \infty,\ \ \frac{m_n}{n^{2\alpha m_n-1}b_n^4}\to\infty,
$$
Theorem \ref{thm1} holds.

In particular, when $\alpha \in (0,1/2m_n)$, we choose the sequences $\{b_n, n\ge 1\}$ and $\{m_n, n\ge 1\}$ satisfy
$$
\frac{b_n^2}{\log n} \to \infty,\ \ \frac{m_n}{b_n^2}\to \infty,\ \ \frac{m_n}{b_n^4}\to\infty,
$$
Theorem \ref{thm1} holds.
\end{exa}
\section{Proofs of main results}
We state some useful lemmas to prove our main results. Firstly,
let us introduce a simplified version of Puhalskii's result \cite{Puh} applied to an array of martingale differences.

\begin{lem}\label{lem1} $($Puhalskii$)$ Let $\left\{d_j^n, 1\le j\le n\right\}$ be a triangular array of martingale differences with values in $\rr^d$, with respect to a filtration $\{\FF_n, n\ge 1\}$. Let $\{b_n, n\ge 1\}$ be a sequence of increasing positive numbers satisfying
$$
b_n\to\infty,\ \ \ \frac{b_n}{\sqrt{n}}\to 0.
$$
Suppose that there
exists a symmetric positive-semidefinite matrix $Q$ such that
\beq\label{lem1-1}
\lim_{n\to\infty}\frac{1}{b_n^2}\log \pp\left(\left|\frac{1}{n}\sum_{k=1}^n\ee\Big[d_k^n(d_k^n)^{'}|\FF_{k-1}\Big]-Q\right|>\varepsilon\right)=-\infty,\ \ \text{for any}\ \varepsilon>0.
\deq
Suppose that there exists a constant $c>0$ such that, for each $1\le k\le n$,
\beq\label{lem1-2}
\Big|d_k^n\Big|\le c\frac{\sqrt{n}}{b_n}\ \ a.s.
\deq
Suppose also that, for all $a>0$, we have the exponential Lindeberg's condition
\beq\label{lem1-3}
\lim_{n\to\infty}\frac{1}{b_n^2}\log \pp\left(\frac{1}{n}\sum_{k=1}^n\ee\left[|d_k^n|^2\ii\left\{|d_k^n|>a\frac{\sqrt{n}}{b_n}\right\}\Big|\FF_{k-1}\right]>\varepsilon\right)=-\infty,\ \ \text{for any}\ \varepsilon>0.
\deq
Then the sequence
$$
\left\{\frac{1}{b_n\sqrt{n}}\sum_{k=1}^n d_k^n,\ n\ge 1\right\}
$$
satisfies the large deviation principle on $\rr^d$ with speed $b_n^2$ and rate function
$$
I(\nu)=\sup_{\lambda\in\rr^d}\left(\lambda^{'}\nu-\frac{1}{2}\lambda^{'}Q\lambda\right).
$$
In particular, if $Q$ is invertible,
$$
I(\nu)=\frac{1}{2}\lambda^{'}Q^{-1}\lambda.
$$
\end{lem}
\begin{proof} The proof of Lemma \ref{lem1} is contained in the proof of Theorem 3.1 in \cite{Puh}.
\end{proof}
\begin{lem}\label{H}\cite{H} Let $X_1, X_2, \cdots, X_n$ be independent random variables with $\ee X_i=0$ and $a_i\le X_i\le b_i$ for any $1\le i\le n$, where $a_1, b_1, a_2, b_2, \cdots, a_n, b_n$ are constants with $a_i<b_i$ for every $1\le i\le n$. Then for any $t>0$, we have
$$
\ee\exp\left(t\sum_{i=1}^n X_i\right) \le \exp\left(\frac{1}{8}t^2\sum_{i=1}^n(b_i-a_i)^2\right).
$$
In particular, for any $r>0$, we have
$$
\pp\left(\left|\sum_{i=1}^n X_i\right|>rn\right) \le 2\exp\left(-\frac{2 n^2r^2}{\sum_{i=1}^n(b_i-a_i)^2}\right).
$$
\end{lem}
\begin{lem}\label{lem-hr} \cite[Theorem 3.4]{H-R}
Let $X_1, X_2, \ldots, X_n$ be independent random variables defined on a probability $(\Omega, \FF, \pp)$.
Let us consider for all integer $n\ge 2$,
$$
U_n=\sum_{i=2}^n \sum_{j=1}^{i-1} g_{i,j}(X_i, X_j),
$$
where the $g_{i,j}: \rr\times\rr\to \rr$ are Borel measurable functions verifying
$$
\ee\left(g_{i,j}(X_i, X_j)|X_i\right)=0\ \ \text{and}\ \ \ee\left(g_{i,j}(X_i, X_j)|X_j\right)=0.
$$
Let $u>0$, $\varepsilon>0$ and let $|g_{i,j}|\le A$ for all $i,j$. Then we have
$$
\pp\Big(U_n\ge 2(1+\varepsilon)^{3/2}C\sqrt{u}+\eta(\varepsilon)Du+\beta(\varepsilon)Bu^{3/2}+\gamma(\varepsilon)Au^2\Big)\le 2.77e^{-u}.
$$
Here
\beq\label{C}
C^2=\sum_{i=2}^n \sum_{j=1}^{i-1}\ee\left(g_{i,j}^2(X_i, X_j)\right),
\deq
\beq\label{D}
\aligned
D=&\sup\left\{\ee\left(\sum_{i=2}^n\sum_{j=1}^{i-1}g_{i,j}(X_i, X_j)a_i(X_i)b_j(X_j)\right): \right. \\
& \ \ \ \ \ \ \ \ \ \ \ \ \ \ \ \ \ \  \ \ \ \ \ \ \ \
\left. \ee\left(\sum_{i=2}^na_i^2(X_i)\right)\le 1, \ \ \ee\left(\sum_{j=1}^{n-1}b_j^2(X_j)\right)\le 1
\right\},
\endaligned
\deq
\beq\label{B}
B^2=\max\left\{\sup_{i,t}\left(\sum_{j=1}^{i-1}\ee\left(g_{i,j}^2(t, X_j)|X_i=t\right)\right),\
\sup_{j,t}\left(\sum_{i=j+1}^{n}\ee\left(g_{i,j}^2(X_i, t)|X_j=t\right)\right)
\right\},
\deq
where
$$
\aligned
\eta(\varepsilon)&=\sqrt{2\kappa}(2+\varepsilon+\varepsilon^{-1}),\\
\beta(\varepsilon)&=e(1+\varepsilon^{-1})^2\kappa(\varepsilon)+\Big[\sqrt{2\kappa}(2+\varepsilon+\varepsilon^{-1})\vee\frac{(1+\varepsilon)^2}{\sqrt{2}}\Big],\\
\gamma(\varepsilon)&=(e(1+\varepsilon^{-1})^2\kappa(\varepsilon))\vee\frac{(1+\varepsilon)^2}{3},\\
\kappa&=4,\\
\kappa(\varepsilon)&=2.5+32\varepsilon^{-1}.
\endaligned
$$
\end{lem}
\begin{proof}[{\bf Proof of Theorem \ref{thm1}}]
From the definition of the standard Pearson chi-square statistic, we have
$$
\ee\chi_n^2=m_n-1,\ \ \ Var(\chi^2_n)=\frac{1}{n}\left(2(m_n-1)(n-1)+\sum_{i=1}^{m_n}\frac{1}{p_n(i)}-m_n^2\right)
$$
and
\beq\label{p1}
\aligned
\chi_n^2-\ee\chi_n^2=&\frac{1}{n}\sum_{k=1}^n\sum_{i=1}^{m_n}\frac{1}{p_n(i)}\ii\{X_{k,n}=i\}-m_n\\
&+\frac{1}{n}\sum_{k\neq l}^n\sum_{i=1}^{m_n}\frac{1}{p_n(i)}\ii\{X_{k,n}=i, X_{l,n}=i\}-(n-1).
\endaligned
\deq
For every $n\ge 1$ and $1\le k\le n$, let us define
$$
T_{k,n}:=\sum_{i=1}^{m_n}\frac{1}{p_n(i)}\ii\{X_{k,n}=i\}-m_n
$$
and
$$
Y_{k,n}:=\frac{\sqrt{2}}{\sqrt{nm_n}}\sum_{l=1}^{k-1}\left(\sum_{i=1}^{m_n}\frac{1}{p_n(i)}\ii_{\{X_{k,n}=i\}}\ii_{\{X_{l,n}=i\}}
-1\right).
$$
Then, for every $n\ge 1$, $\{T_{k,n}, 1\le k\le n\}$ is a sequence of independent and identically distributed random variables, and $\{Y_{k,n}, 1\le k\le n\}$ is a sequence of martingale difference with respect to $\sigma$-algebra $\FF_{k,n}=\sigma(X_{i,n}, i\le k)$. Hence we have
$$
\chi_n^2-\ee\chi_n^2=\frac{1}{n}\sum_{k=1}^n T_{k,n}+\frac{\sqrt{2m_n}}{\sqrt{n}}\sum_{k=2}^nY_{k,n}.
$$

In order to prove the desired result, it is enough to show that for any $\varepsilon>0$,
\beq\label{cl-1}
\lim_{n\to\infty}\frac{1}{b_n^2}\log \pp\left(\frac{1}{b_n\sqrt{2m_n}}\left|\frac{1}{n}\sum_{k=1}^n T_{k,n}\right|>\varepsilon\right)=-\infty
\deq
and
for any $r>0$,
\beq\label{p6}
\lim_{n\to\infty}\frac{1}{b_n^2}\log \pp\left(\frac{1}{b_n\sqrt{n}}\left|\sum_{k=2}^n Y_{k,n}\right|>r\right)=-\frac{r^2}{2}.
\deq
It is easy to see that
\beq\label{bbn}
-m_n\le\sum_{i=1}^{m_n}\frac{1}{p_n(i)}\ii\{X_{k,n}=i\}
-m_n\le \frac{1}{p_{min}}-m_n.
\deq
From Lemma \ref{H} and the condition (\ref{con}),  we have
\begin{align*}
&\lim_{n\to\infty}\frac{1}{b_n^2}\log \pp\left(\frac{1}{b_n\sqrt{2m_n}}\left|\frac{1}{n}\sum_{k=1}^n T_{k,n}\right|>\varepsilon\right)\\
=&\lim_{n\to\infty}\frac{1}{b_n^2}\log \pp\left(\frac{1}{b_n\sqrt{2m_n}}\left|\frac{1}{n}\sum_{k=1}^n\left(\sum_{i=1}^{m_n}\frac{1}{p_n(i)}\ii\{X_{k,n}=i\}
-m_n\right)\right|>\varepsilon\right)\\
\le&\lim_{n\to\infty}\frac{1}{b_n^2}\log \left(2\exp\left(-\frac{4\varepsilon^2 b_n^2m_nn^2p_{min}^2}{n}\right)\right)\\
=&-\lim_{n\to\infty}\frac{4\varepsilon^2 m_nnp_{min}^2}{b_n^2}=-\infty,
\end{align*}
which yields the claim (\ref{cl-1}). Furthermore,
for every $n\ge 1$ and $1\le k\le n$, let us define
$$
Y_{k,n}=Z_{k,n}+Z_{k,n}^{'}
$$
where
$$
Z_{k,n}:=Y_{k,n}\ii\left\{|Y_{k,n}|\le\frac{\sqrt{n}}{b_n}\right\}-\ee\left(Y_{k,n}\ii\left\{|Y_{k,n}|\le\frac{\sqrt{n}}{b_n}\right\}\Big|\FF_{k-1,n}\right)
$$
and
$$
Z_{k,n}^{'}:=Y_{k,n}\ii\left\{|Y_{k,n}|>\frac{\sqrt{n}}{b_n}\right\}-\ee\left(Y_{k,n}\ii\left\{|Y_{k,n}|>\frac{\sqrt{n}}{b_n}\right\}\Big|\FF_{k-1,n}\right).
$$
Obviously $\{Z_{k,n}, \FF_{k,n}, 1\le k\le n\}$ and $\{Z^{'}_{k,n}, \FF_{k,n}, 1\le k\le n\}$ are also sequences of martingale differences and $\dsp |Z_{k,n}|\leq 2\frac{\sqrt{n}}{b_n}\ a.s$. Hence, in order to obtain (\ref{p6}), it is enough to prove that for any $r>0$,
\beq\label{p8}
\lim_{n\to\infty}\frac{1}{b_n^2}\log \pp\left(\frac{1}{b_n\sqrt{n}}\left|\sum_{k=2}^n Z_{k,n}\right|>r\right)=-\frac{r^2}{2}
\deq
and for any $\varepsilon>0$,
\beq\label{p9}
\lim_{n\to\infty}\frac{1}{b_n^2}\log \pp\left(\frac{1}{b_n\sqrt{n}}\left|\sum_{k=2}^n Z_{k,n}^{'}\right|>\varepsilon\right)=-\infty.
\deq

Firstly, for any $1\le l\le n$ and $1\le j\le m_n$, it is easy to see that
$$
-1\le \frac{1}{p_n(j)}\ii\{X_{l,n}=j\}-1\le\frac{1}{p_{min}}-1.
$$
Then, from Lemma \ref{H}, for any $a>0$, we have
\begin{align*}
&  \sum_{k=2}^n \pp\left(|Y_{k,n}|>a\frac{\sqrt{n}}{b_n}\right)\\
=& \sum_{k=2}^n\pp\left(\left|\frac{\sqrt{2}}{\sqrt{nm_n}}\sum_{l=1}^{k-1}\left(\sum_{i=1}^{m_n}\frac{1}{p_n(i)}\ii\{X_{k,n}=i, X_{l,n}=i\}-1\right)\right|>a\frac{\sqrt{n}}{b_n}\right)\\
=& \sum_{k=2}^n\pp\left(\left|\sum_{l=1}^{k-1}\left(\sum_{i=1}^{m_n}\frac{1}{p_n(i)}\ii\{X_{k,n}=i, X_{l,n}=i\}-1\right)\right|>a\frac{n\sqrt{m_n}}{\sqrt{2}b_n}\right)\\
=& \sum_{k=2}^n\sum_{j=1}^{m_n}\pp\left(\left|\sum_{l=1}^{k-1}\left(\sum_{i=1}^{m_n}\frac{1}{p_n(i)}\ii\{X_{k,n}=i, X_{l,n}=l\}-1\right)\right|>a\frac{n\sqrt{m_n}}{\sqrt{2}b_n}, X_{k,n}=j\right)\\
=&  \sum_{k=2}^n\sum_{j=1}^{m_n}p_n(j)\pp\left(\left|\sum_{l=1}^{k-1}\left(\frac{1}{p_n(j)}\ii\{X_{l,n}=j\}-1\right)\right|>a\frac{n\sqrt{m_n}}{\sqrt{2}b_n}\right)\\
\le&   2\sum_{k=2}^n\sum_{j=1}^{m_n}p_n(j)\exp\left\{-\frac{a^2n^2m_np_{min}^2}{2(k-1)b_n^2}\right\}\le   2\sum_{k=2}^n\sum_{j=1}^{m_n}p_n(j)\exp\left\{-\frac{a^2nm_np_{min}^2}{2b_n^2}\right\}\\
\le &2n\exp\left\{-\frac{a^2nm_np_{min}^2}{2b_n^2}\right\},
\end{align*}
which, together with the conditions $\log n/b_n^2\to 0$ and $\dsp \frac{n m_n p_{min}^2}{b_n^4}\to\infty$, implies that
\beq\label{p10}
\lim_{n\to\infty}\frac{1}{b_n^2}\log \sum_{k=2}^n \pp\left(|Y_{k,n}|>a\frac{\sqrt{n}}{b_n}\right)=-\infty.
\deq

For any $\varepsilon>0$, since
\begin{align*}
&\pp\left(\frac{1}{b_n\sqrt{n}}\left|\sum_{k=2}^n Z_{k,n}^{'}\right|>\varepsilon\right)\\
\le &\pp\left(\frac{1}{b_n\sqrt{n}}\left|\sum_{k=2}^n Y_{k,n}\ii\left\{|Y_{k,n}|>\frac{\sqrt{n}}{b_n}\right\}\right|>\frac{\varepsilon}{2}\right)\\
&\ \ \ \ \ \ +\pp\left(\frac{1}{b_n\sqrt{n}}\left|\sum_{k=2}^n \ee\left(Y_{k,n}\ii\left\{|Y_{k,n}|>\frac{\sqrt{n}}{b_n}\right\}\Big|\FF_{k-1,n}\right)
\right|>\frac{\varepsilon}{2}\right)\\
\le & 2\sum_{k=2}^n \pp\left(|Y_{k,n}|>\frac{\sqrt{n}}{b_n}\right),
\end{align*}
then, from (\ref{p10}), we get the claim (\ref{p9}).

 Next, since $\dsp |Z_{k,n}|\leq 2\frac{\sqrt{n}}{b_n}\ a.s$ for each $1\le k\le n$, in order to obtain (\ref{p8}), by Lemma \ref{lem1}, it is enough to prove that for any $\varepsilon>0$,
\beq\label{p11}
\lim_{n\to\infty}\frac{1}{b_n^2}\log \pp\left(\left|\frac{1}{n}\sum_{k=2}^n\ee(Z_{k,n}^2|\FF_{k-1,n})-1\right|>\varepsilon\right)=-\infty
\deq
and for all $a>0$,
\beq\label{p12}
\lim_{n\to\infty}\frac{1}{b_n^2}\log\pp\left(\frac{1}{n}\sum_{k=2}^n
\ee\left(Z_{k,n}^2\ii\left\{|Z_{k,n}|>a\frac{\sqrt{n}}{b_n}\right\}\Big|\FF_{k-1,n}\right)>\varepsilon\right)=-\infty.
\deq
For each $1\le k\le n$, from the fact
$$
\ee\left(Y_{k,n}\ii\left\{|Y_{k,n}|\le\frac{\sqrt{n}}{b_n}\right\}\Big|\FF_{k-1,n}\right)
=-\ee\left(Y_{k,n}\ii\left\{|Y_{k,n}|>\frac{\sqrt{n}}{b_n}\right\}\Big|\FF_{k-1,n}\right),
$$
then for any $\varepsilon>0$, we have
\begin{align*}
& \pp\left(\frac{1}{n}\sum_{k=2}^n
\ee\left(Z_{k,n}^2\ii\left\{|Z_{k,n}|>a\frac{\sqrt{n}}{b_n}\right\}\Big|\FF_{k-1,n}\right)>\varepsilon\right)\\
\leq& \sum_{k=2}^n\pp\left(
\ee\left(Z_{k,n}^2\ii\left\{|Z_{k,n}|>a\frac{\sqrt{n}}{b_n}\right\}\Big|\FF_{k-1,n}\right)>\varepsilon\right)\\
\leq&
 \sum_{k=2}^{n}\pp\left(|Z_{k,n}|>a\frac{\sqrt{n}}{b_n}\right)\\
\leq&\sum_{k=2}^n\pp\left(\left|Y_{k,n}\ii\left\{|Y_{k,n}|\le\frac{\sqrt{n}}{b_n}\right\}
\right|>\frac{a\sqrt{n}}{2b_n}\right)\\
&\ \ \ +\sum_{k=2}^n\pp\left(\left|\ee\left(Y_{k,n}\ii\left\{|Y_{k,n}|\le\frac{\sqrt{n}}{b_n}\right\}
\Big|\FF_{k-1,n}\right)\right|>\frac{a\sqrt{n}}{2b_n}\right)\\
\leq&\sum_{k=2}^n\pp\left(|Y_{k,n}|>\frac{a\sqrt{n}}{2b_n}\right)+\sum_{k=2}^n\pp\left(|Y_{k,n}|>\frac{\sqrt{n}}{b_n}\right).
\end{align*}
Hence from (\ref{p10}), the claim (\ref{p12}) holds.

In order show the calim (\ref{p11}), it is enough to prove that for any $\varepsilon>0$,
\beq\label{p18-1}
\lim_{n\to\infty}\frac{1}{b_n^2}\log \pp\left(\left|\frac{1}{n}\sum_{k=2}^n\Big[\ee(Z_{k,n}^2|\FF_{k-1,n})-\ee(Y_{k,n}^2|\FF_{k-1,n})\Big]\right|>\varepsilon\right)=-\infty
\deq
and
\beq\label{p18}
\lim_{n\to\infty}\frac{1}{b_n^2}\log \pp\left(\left|\frac{1}{n}\sum_{k=2}^n\ee(Y_{k,n}^2|\FF_{k-1,n})-1\right|>\varepsilon\right)=-\infty.
\deq
For each $1\le k\le n$, since
\begin{align*}
&\ee(Z_{k,n}^2|\FF_{k-1,n})-\ee(Y_{k,n}^2|\FF_{k-1,n})\\
=&\ee\left(Y^2_{k,n}\ii\left\{|Y_{k,n}|\le\frac{\sqrt{n}}{b_n}\right\}
\Big|\FF_{k-1,n}\right)-\ee(Y_{k,n}^2|\FF_{k-1,n})\\
&\ \ \ \ -\left(\ee\left(Y_{k,n}\ii\left\{|Y_{k,n}|\le\frac{\sqrt{n}}{b_n}\right\}\Big|\FF_{k-1,n}\right)
\right)^2\\
=&-\ee\left(Y^2_{k,n}\ii\left\{|Y_{k,n}|>\frac{\sqrt{n}}{b_n}\right\}
\Big|\FF_{k-1,n}\right)\\
&\ \ \ \ -\left(\ee\left(Y_{k,n}\ii\left\{|Y_{k,n}|>\frac{\sqrt{n}}{b_n}\right\}\Big|\FF_{k-1,n}\right)
\right)^2,
\end{align*}
then we have
\begin{align*}
&\pp\left(\left|\frac{1}{n}\sum_{k=2}^n\Big[\ee(Z_{k,n}^2|\FF_{k-1,n})-\ee(Y_{k,n}^2|\FF_{k-1,n})\Big]\right|>\varepsilon\right)\\
\le & \pp\left(\frac{1}{n}\sum_{k=2}^n\ee\left(Y^2_{k,n}\ii\left\{|Y_{k,n}|>\frac{\sqrt{n}}{b_n}\right\}
\Big|\FF_{k-1,n}\right)>\frac{\varepsilon}{2}\right)\\
&\ \ \ \ \ +\pp\left(\frac{1}{n}\sum_{k=2}^n\left(\ee\left(Y_{k,n}\ii\left\{|Y_{k,n}|>\frac{\sqrt{n}}{b_n}\right\}\Big|\FF_{k-1,n}\right)
\right)^2>\frac{\varepsilon}{2}\right)\\
\le & 2\sum_{k=2}^n \pp\left(|Y_{k,n}|>\frac{\sqrt{n}}{b_n}\right).
\end{align*}
Hence from (\ref{p10}), the claim (\ref{p18-1}) holds.

For each $1\le k\le n$, since
\beq\label{p20}
\aligned
&\ee(Y_{k,n}^2|\FF_{k-1,n})\\
=&\frac{2}{nm_n}\sum_{s\ne l}^{k-1}\ee\left[\left(\sum_{i=1}^{m_n}\frac{1}{p_n(i)}\ii_{\{X_{k,n}=i\}}\ii_{\{X_{l,n}=i\}}
-1\right)\left(\sum_{i=1}^{m_n}\frac{1}{p_n(i)}\ii_{\{X_{k,n}=i\}}\ii_{\{X_{s,n}=i\}}
-1\right)\Big|\FF_{k-1,n}\right]\\
&\ \ \ +\frac{2}{nm_n}\sum_{l=1}^{k-1}\ee\left[\left(\sum_{i=1}^{m_n}\frac{1}{p_n(i)}\ii_{\{X_{k,n}=i\}}\ii_{\{X_{l,n}=i\}}
-1\right)^2\Big|\FF_{k-1,n}\right]\\
=&\frac{2}{nm_n}\sum_{s\ne l}^{k-1}\left(\sum_{i=1}^{m_n}\frac{1}{p_n(i)}\ii\{X_{s,n}=i\}\ii\{X_{l,n}=i\}-1\right)\\
&\ \ \ +\frac{2}{nm_n}\sum_{l=1}^{k-1}\left(\sum_{i=1}^{m_n}\frac{1}{p_n(i)}\ii\{X_{l,n}=i\}-1\right),
\endaligned
\deq
then, in order to prove (\ref{p18}), it is enough to show that the following claims hold: for any $\varepsilon > 0$,
\beq\label{p21}
\lim_{n\to\infty}\frac{1}{b_n^2}\log \pp\left(\frac{2}{n^2m_n}\left|\sum_{k=2}^n\sum_{s\ne l}^{k-1}\left(\sum_{i=1}^{m_n}
\frac{1}{p_n(i)}\ii\{X_{s,n}=i\}\ii\{X_{l,n}=i\}-1\right)
\right|>\varepsilon\right)=-\infty
\deq
and
\beq\label{p22}
\lim_{n\to\infty}\frac{1}{b_n^2}\log\pp\left(\left|\frac{2}{n^2m_n}\sum_{k=2}^n\sum_{l=1}^{k-1}
\left(\sum_{i=1}^{m_n}\frac{1}{p_n(i)}\ii\{X_{l,n}=i\}-1\right)-1\right|>\varepsilon\right)
=-\infty.
\deq

Firstly, for every $n\ge 1$ and $1\le k\le n$, let us define
$$
U_{k,n}=\sum_{s=2}^{k-1} \sum_{l=1}^{s-1} g_{s,l}(X_{s,n}, X_{l,n}),
$$
where
$$
\aligned
g_{s,l}(X_{s,n}, X_{l,n})=&
\sum_{i=1}^{m_n}
\frac{1}{p_n(i)}\ii\{X_{s,n}=i\}\ii\{X_{l,n}=i\}-1\\
=&\sum_{i=1}^{m_n}
\frac{1}{p_n(i)}\left(\ii\left\{X_{s, n}=i\right\}-p_n(i)\right)\left(\ii\left\{X_{l, n}=i\right\}-p_n(i)\right).
\endaligned
$$
It is easy to check that $|g_{s,l}(X_{s,n}, X_{l,n})|\le \frac{1}{p_{min}}=:A$ and
$$
\ee\left(g_{s,l}(X_{s,n}, X_{l,n})|X_{s,n}\right)= \ee\left(g_{s,l}(X_{s,n}, X_{l,n})|X_{l,n}\right)=0,\ \ \ \ \ s\ne l.
$$
Now we shall estimate the parameters $B, C, D$ in Lemma \ref{lem-hr}. It is easy to check that
\begin{align*}
g_{s,l}^2(X_{s,n}, X_{l,n})=& \left(\sum_{i=1}^{m_n}\frac{1}{p_n(i)}\left(\ii\{X_{s,n}=i\}-p_n(i)\right)\left(\ii\{X_{l,n}=i\}-p_n(i)\right)\right)^2\\
=&\sum_{i=1}^{m_n}\frac{1}{p_n^2(i)}\left(\ii\{X_{s,n}=i\}-p_n(i)\right)^2\left(\ii\{X_{l,n}=i\}-p_n(i)\right)^2\\
&\ +\sum_{i\ne i^{'}}^{m_n}\frac{1}{p_n(i)p_n(i^{'})}\left(\ii\{X_{s,n}=i\}-p_n(i)\right)\left(\ii\{X_{l,n}=i\}-p_n(i)\right)\\
&\ \ \ \ \ \ \ \ \ \ \ \ \ \cdot\left(\ii\{X_{s,n}=i^{'}\}-p_n(i^{'})\right)\left(\ii\{X_{l,n}=i^{'}\}-p_n(i^{'})\right),
\end{align*}
\begin{align*}
\ee\Big[\left(\ii\{X_{s,n}=i\}-p_n(i)\right)^2\left(\ii\{X_{l,n}=i\}-p_n(i)\right)^2\Big]=p_n^2(i)(1-p_n(i))^2
\end{align*}
and
\begin{align*}
&\ee\Big[\left(\ii\{X_{s,n}=i\}-p_n(i)\right)\left(\ii\{X_{l,n}=i\}-p_n(i)\right)\\
&\ \ \ \  \left(\ii\{X_{s,n}=i^{'}\}-p_n(i^{'})\right)\left(\ii\{X_{l,n}=i^{'}\}-p_n(i^{'})\right)\Big]\\
=&\ee\left(\ii\{X_{s,n}=i\}\ii\{X_{s,n}=i^{'}\}-p_n(i)\ii\{X_{s,n}=i^{'}\}-\ii\{X_{s,n}=i\}p_n(i^{'})+p_n(i^{'})p_n(i)\right)\\
&\ \ \cdot\ee\left(\ii\{X_{l,n}=i\}\ii\{X_{l,n}=i^{'}\}-p_n(i)\ii\{X_{l,n}=i^{'}\}-\ii\{X_{l,n}=i\}p_n(i^{'})+p_n(i^{'})p_n(i)\right)\\
=&\ee\left(-p_n(i)\ii\{X_{s,n}=i^{'}\}-\ii\{X_{s,n}=i\}p_n(i^{'})+p_n(i^{'})p_n(i)\right)\\
&\ \ \cdot\ee \left(-p_n(i)\ii\{X_{l,n}=i^{'}\}-\ii\{X_{l,n}=i\}p_n(i^{'})+p_n(i^{'})p_n(i)\right)\\
=&p_n^2(i^{'})p_n^2(i).
\end{align*}
Hence we have
\begin{align*}
\ee\left(g_{s,l}^2(X_{s,n}, X_{l,n})\right)=&\sum_{i=1}^{m_n}\left(1-p_n(i)\right)^2+\sum_{i\ne i^{'}}^{m_n}p_n(i)p_n(i^{'})\\
=& m_n-2+\sum_{i=1}^{m_n}p_n^2(i)+\sum_{i\ne i^{'}}^{m_n}p_n(i)p_n(i^{'})=m_n-1,
\end{align*}
which yields that
$$
C^2=\sum_{s=2}^{k-1} \sum_{l=1}^{s-1}\ee\left(g^2_{s,l}(X_{s,n}, X_{l,n})\right)\le m_nk^2\le m_nn^2.
$$

Furthermore, for any $1\le t\le m_n$, we have
\begin{align*}
&\ee\left(g_{s,l}^2(t, X_{l,n})\Big| X_{s,n}=t\right)\\
=& \ee\left(\left(\sum_{i=1}^{m_n}
\frac{1}{p_n(i)}\ii\{X_{s,n}=i\}\ii\{X_{l,n}=i\}-1\right)^2\Big|X_{s,n}=t\right)\\
=&\ee\left(\left(\frac{1}{p_n(t)}\ii\{X_{s,n}=t\}\ii\{X_{l,n}=t\}-1\right)^2\Big|X_{s,n}=t\right)\\
=& \frac{1}{p_n(t)}\ii\{X_{s,n}=t\}-2\ii\{X_{s,n}=t\}+1\le \frac{1}{p_{min}},
\end{align*}
which implies that
$$
B^2\le \frac{n}{p_{min}}.
$$

For every $1\le k\le n$, assume that
$$
\ee\left(\sum_{s=2}^{k-1}a_s^2(X_{s,n})\right)\le 1\ \ \text{and}\ \  \ee\left(\sum_{l=1}^{k-2}b_l^2(X_{l,n})\right)\le 1,
$$
then, by using Cauchy-Schwarz inequality and Jensen's inequality, we have
\beq\label{p25}
\aligned
&\sum_{s=2}^{k-1}\sum_{l=1}^{s-1}(\ee a ^2_s(X_{s,n}))^{1/2}
\left(\ee b ^2_l(X_{l,n})\right)^{1/2}\\
\le&\sum_{s=2}^{k-1}(\ee a ^2_s(X_{s,n}))^{1/2}
\sum_{l=1}^{k-2}\left(\ee b ^2_l(X_{l,n})\right)^{1/2}\\
\le&n\left(\left(\sum_{s=2}^{k-1}\ee a_s^2(X_{s,n})\right)\cdot\left(\sum_{l=1}^{k-2}\ee b_l^2(X_{l,n})\right)\right)^{1/2}\le n.
\endaligned
\deq
Moreover, from the definition of $g_{s,l}(X_{s,n}, X_{l,n})$, it follows that
\begin{align*}
&g_{s,l}(X_{s,n}, X_{l,n})a_s(X_{s,n})b_l(X_{l,n})\\
=&\sum_{i=1}^{m_n} \frac{1}{p_n(i)}\ii\left\{X_{s, n}=i\right\}\ii\left\{X_{l, n}=i\right\}a_s(X_{s,n})b_l(X_{l,n})
-a_s(X_{s,n})b_l(X_{l,n}).
\end{align*}
By using Cauchy-Schwarz inequality and (\ref{p25}), we have
\begin{align*}
&\left|\ee\left(\sum_{s=2}^{k-1}\sum_{l=1}^{s-1}\sum_{i=1}^{m_n} \frac{1}{p_n(i)}\ii\left\{X_{s, n}=i\right\}\ii\left\{X_{l, n}=i\right\}a_s(X_{s,n})b_l(X_{l,n})
\right)\right|\\
=&\left|\sum_{s=2}^{k-1}\sum_{l=1}^{s-1}\sum_{i=1}^{m_n} \frac{1}{p_n(i)}\ee\left(\ii\left\{X_{s, n}=i\right\}a_s(X_{s,n})\right)\ee\left(\ii\left\{X_{l, n}=i\right\}b_l(X_{l,n})
\right)\right|\\
\le&\sum_{s=2}^{k-1}\sum_{l=1}^{s-1}\left(\sum_{i=1}^{m_n} \frac{1}{p_n(i)}\Big[\ee\left(\ii\left\{X_{s, n}=i\right\}a_s(X_{s,n})\right)\Big]^2\right)^{1/2}\\
&\ \ \ \ \ \ \ \ \ \ \left(\sum_{i=1}^{m_n} \frac{1}{p_n(i)}\Big[\ee\left(\ii\left\{X_{l, n}=i\right\}b_l(X_{l,n})
\right)\Big]^2\right)^{1/2}\\
\le&\sum_{s=2}^{k-1}\sum_{l=1}^{s-1}\left(\sum_{i=1}^{m_n} \frac{1}{p_n(i)}\Big[(p_n(i))^{1/2}\left(\ee(a ^2_s(X_{s,n}))^{1/2}\right)\Big]^2\right)^{1/2}\\
&\ \ \ \ \ \ \ \ \ \ \left(\sum_{i=1}^{m_n}
\frac{1}{p_n(i)}\Big[(p_n(i))^{1/2}\left(\ee(b ^2_l(X_{l,n}))^{1/2}\right)\Big]^2\right)^{1/2}\\
\le&
\sum_{s=2}^{k-1}\sum_{l=1}^{s-1}(\ee a ^2_s(X_{s,n}))^{1/2}
\left(\ee b ^2_l(X_{l,n})\right)^{1/2}\le n
\end{align*}
and
$$
\aligned
&\left|\ee\left(\sum_{s=2}^{k-1}\sum_{l=1}^{s-1}a_s(X_{s,n})b_l(X_{l,n})\right)\right|\\
\le&\sum_{s=2}^{k-1}\sum_{l=1}^{s-1}(\ee a ^2_s(X_{s,n}))^{1/2}
\left(\ee b ^2_l(X_{l,n})\right)^{1/2}\le n,
\endaligned
$$
which implies
$$
D\le 2n.
$$
Let us define
\beq\label{delta}
\aligned
\Delta_n:&=2(1+\varepsilon)^{3/2}C\sqrt{u_n}+\eta(\varepsilon)Du_n+\beta(\varepsilon)Bu_n^{3/2}+\gamma(\varepsilon)Au_n^2\\
=&O\left(n\sqrt{m_n}\sqrt{u_n}+nu_n+\frac{\sqrt{n}}{\sqrt{p_{min}}}u_n^{3/2}+\frac{u_n^2}{p_{min}}\right).
\endaligned
\deq
From the condition $\dsp m_n/b_n^2\to\infty$, we can choose $l_n$, which is a sequence of positive numbers such that
$$
l_n\to\infty,\ \ \frac{m_n}{b_n^2l_n^3}\to\infty.
$$
By taking $u_n=b_n^2l_n$, we get
$$
\aligned
\Delta_n=&O\left(n\sqrt{m_n}b_n\sqrt{l_n}+nb_n^2l_n+\frac{\sqrt{n}}{\sqrt{p_{min}}}b_n^3l_n^{3/2}+\frac{b_n^4l_n^2}{p_{min}}\right).
\endaligned
$$
From the conditions in (\ref{con}) and noting the fact $p_{min}^{-1}\ge m_n$, it is easy to check
$$
\frac{nm_n}{n\sqrt{m_n}b_n\sqrt{l_n}}
=\left(\frac{m_n}{b_n^2l_n}\right)^{1/2}\to\infty,
$$
$$
\frac{nm_n}{nb_n^2l_n}
=\frac{m_n}{b_n^2l_n}
\to\infty,
$$
$$
\frac{nm_n\sqrt{p_{min}}}{\sqrt{n}b_n^3l_n^{3/2}}
=\frac{\sqrt{nm_n}p_{min}}{b_n^2}\frac{\sqrt{m_n}}{\sqrt{p_{min}}b_nl_n^{3/2}}\to\infty
$$
and
$$
\frac{nm_n p_{min}}{b_n^4l_n^2}=\frac{nm_n p_{min}^2}{b_n^4}\frac{1}{p_{min}l_n^2}\ge \frac{nm_n p_{min}^2}{b_n^4}\frac{m_n}{l_n^2}\to\infty
$$
which implies
$$
\aligned
\Delta_n=o\left(nm_n\right).
\endaligned
$$
Therefore, by using Lemma \ref{lem-hr}, for any $\varepsilon>0$, we have
\begin{align*}
&\lim_{n\to\infty}\frac{1}{b_n^2}\log \pp\left(\frac{2}{n^2m_n}\left|\sum_{k=2}^n\sum_{s=2 }^{k-1}\sum_{l=1}^{s-1}\left(\sum_{i=1}^{m_n}
\frac{1}{p_n(i)}\ii\{X_{s,n}=i\}\ii\{X_{l,n}=i\}-1\right)
\right|>\varepsilon\right)\\
=&\lim_{n\to\infty}\frac{1}{b_n^2}\log \sum_{k=2}^n \pp\left(\left|U_k\right|>\frac{nm_n\varepsilon}{2}\right)\\
\le &\lim_{n\to\infty}\frac{1}{b_n^2}\log \sum_{k=2}^n\pp\left(\left|U_k\right|>\Delta_n\right)\\
\le &-\lim_{n\to\infty}\frac{u_n-\log n}{b_n^2}\to-\infty,
\end{align*}
which is the claim (\ref{p21}).

At last, for any $\varepsilon>0$, by the conditions in (\ref{con}), the bound in (\ref{bbn}) and Lemma \ref{H}, we have
\begin{align*}
&\lim_{n\to\infty}\frac{1}{b_n^2}\log\pp\left(\left|\frac{2}{n^2m_n}\sum_{k=2}^n\sum_{l=1}^{k-1}
\left(\sum_{i=1}^{m_n}\frac{1}{p_n(i)}\ii\{X_{l,n}=i\}-1\right)-1\right|>\varepsilon\right)
\\
\le &\lim_{n\to\infty}\frac{1}{b_n^2}\log\pp\left(\left|\frac{2}{n^2m_n}\sum_{k=2}^n\sum_{l=1}^{k-1}
\sum_{i=1}^{m_n}\left(\frac{1}{p_n(i)}\ii\{X_{l,n}=i\}-1\right)\right|+\frac{n+m_n-1}{nm_n}>\varepsilon\right)\\
\le &\lim_{n\to\infty}\frac{1}{b_n^2}\log \sum_{k=2}^n\pp\left(\left|\frac{2}{nm_n}\sum_{l=1}^{k-1}
\sum_{i=1}^{m_n}\left(\frac{1}{p_n(i)}\ii\{X_{l,n}=i\}-1\right)\right|>\varepsilon-\frac{n+m_n-1}{nm_n}\right)
\\
\le&\lim_{n\to\infty}\frac{1}{b_n^2}\log\sum_{k=2}^n \left( 2\exp\left(-\frac{\varepsilon^2}{8}nm_n^2p_{min}^2\right)\right)\\
=&-\lim_{n\to\infty}\frac{8^{-1}\varepsilon^2nm_n^2p_{min}^2+\log n}{b_n^2}=-\infty,
\end{align*}
which implies the claim (\ref{p22}).
\end{proof}

\end{document}